\documentclass{amsart}

\usepackage{amsmath,amssymb}
\usepackage{amsthm}
\usepackage{amsrefs}
\usepackage{indentfirst}
\usepackage{mathrsfs}
\usepackage{hyperref}
\usepackage{xcolor}
\usepackage[margin=1.4in]{geometry}

\usepackage{mathtools}
\usepackage{algorithm}
\usepackage{algorithmicx}
\usepackage{algpseudocode}
\usepackage{enumitem}

\newlength{\leftstackrelawd}
\newlength{\leftstackrelbwd}
\def\leftstackrel#1#2{\settowidth{\leftstackrelawd}%
    {${{}^{#1}}$}\settowidth{\leftstackrelbwd}{$#2$}%
    \addtolength{\leftstackrelawd}{-\leftstackrelbwd}%
    \leavevmode\ifthenelse{\lengthtest{\leftstackrelawd>0pt}}%
    {\kern-.5\leftstackrelawd}{}\mathrel{\mathop{#2}\limits^{#1}}}

\numberwithin{equation}{section}

\theoremstyle{plain}
\newtheorem{maintheorem}{Theorem}
\newtheorem{mainassumption}{Assumption}
\newtheorem{maincorollary}[maintheorem]{Corollary}
\newtheorem{theorem}{Theorem}[section]
\newtheorem{lemma}[theorem]{Lemma}
\newtheorem{proposition}[theorem]{Proposition}

\newtheorem{definition}[theorem]{Definition}
\newtheorem{assumption}[theorem]{Assumption}
\newtheorem{remark}[theorem]{Remark}

\newcommand{\bR}{\mathbb{R}}

\newcommand{\bN}{\mathbb{N}}
\newcommand{\bZ}{\mathbb{Z}}
\newcommand{\bE}{\mathbb{E}}
\newcommand{\bP}{\mathbb{P}}

\newcommand{\crits}{\mathrm{Crit}_s(f)}
\newcommand{\crit}{\mathrm{Crit}(f)}
\newcommand{\norm}[1]{\left\Vert #1\right\Vert}

\newcommand{\trans}{^\top}

\title[RCGD almost surely escapes strict saddle points]
{Randomized coordinate gradient descent almost surely escapes strict
saddle points}

\author{Ziang Chen}
\address{(ZC) Department of Mathematics, Massachusetts Institute of
Technology, Cambridge, MA 02139, USA}
\email{ziang@mit.edu}
\author{Yingzhou Li}
\address{(YL) School of Mathematical Sciences, Fudan University, Shanghai
200433, China}
\email{yingzhouli@fudan.edu.cn}
\author{Zihao Li}
\address{(ZL) School of Mathematical Sciences, Fudan University, Shanghai
200433, China}
\email{zhli24@m.fudan.edu.cn}
\date{\today}

\thanks{The work of ZC is supported in part by the National Science
Foundation via grant DMS-2509011. The work of YL and ZL is supported by
the National Natural Science Foundation of China (NSFC) under grant
numbers 12271109, the Science and Technology Commission of Shanghai
Municipality (STCSM) under grant numbers 22TQ017 and 24DP2600100, and the
Shanghai Institute for Mathematics and Interdisciplinary Sciences (SIMIS)
under grant number SIMIS-ID-2024-(CN). We thank Jianfeng Lu for helpful
discussions.}

\begin{document}

\begin{abstract}
We analyze the behavior of randomized coordinate gradient descent for nonconvex optimization, proving that under standard assumptions, the iterates almost surely escape strict saddle points. By formulating the method as a nonlinear random dynamical system and characterizing neighborhoods of critical points, we establish this result through the center-stable manifold theorem.

\end{abstract}

\maketitle

\section{Introduction}

Randomized coordinate gradient descent is a widely used optimization
method in scientific computing, particularly for large-scale problems. In
this paper, we analyze the escape behavior of randomized coordinate
gradient descent from strict saddle points when applied to smooth,
nonconvex optimization problems of the form
\begin{equation}\label{obj-fun}
\min_{x \in \mathbb{R}^d} f(x).
\end{equation}
Specifically, we study coordinate gradient descent with a randomized
coordinate selection rule and a fixed step size, as detailed in
Algorithm~\ref{alg:RCGD}.

\begin{algorithm}[H]
    \caption{Randomized coordinate gradient descent}
    \label{alg:RCGD}
    \begin{algorithmic}
        \State Initialization: $x_0\in\bR^d$, $t=0$.
        \While {not convergent}
        \State Sample a coordinate $i_t$ uniformly random from $\{1,2,\dots,d\}$.
        \State $x_{t+1}\leftarrow x_t- \alpha e_{i_t} \partial_{i_t} f(x_t)$.
        \State $t\leftarrow t+1$.
        \EndWhile
    \end{algorithmic}
\end{algorithm}

The main contribution of this paper, presented in Theorem~\ref{thm:main},
establishes that under standard assumptions---namely, the smoothness of
the objective function \( f \), the boundedness of its Hessian, and the
non-degeneracy of its strict saddle points---the set of all pairs \((x_0,
\omega)\) converging to a strict saddle point has measure zero. Here, \(
x_0 \) denotes the initial point, and \( \omega \) represents the sequence
of sampled coordinates. Furthermore, under the additional assumption that
all critical points are isolated, the algorithm is guaranteed to converge
globally to a critical point without negative Hessian eigenvalues. The
proof is based on a random dynamical systems perspective, following
\cite{chen2023global}, and the application of the center-stable manifold
theorem that rigorously characterizes the local behavior of the algorithm
near saddle points.

\subsection{Related works}

Coordinate gradient descent (CGD) is a widely used optimization technique,
particularly well-suited for modern large-scale
problems~\cites{wright2015coordinate, shi2016primer}. The method's
popularity stems from its computational efficiency and scalability, making
it applicable across diverse domains~\cite{zhang2025parallel,
wang2019coordinate, wang2023coordinate}. Notable applications include
symmetric eigenvalue problems~\cite{bai2025greedy}, root-finding
algorithms~\cite{tran2025randomized}, quantum circuit
optimization~\cite{ding2024random}, and high-dimensional statistics, where
it is implemented in packages such as
SparseNet~\cite{mazumder2011sparsenet}.

Coordinate gradient descent (CGD) methods fall into two main categories: deterministic approaches that follow fixed coordinate selection rules, and randomized methods that employ stochastic sampling. Among deterministic variants, the cyclic strategy updates coordinates in a fixed periodic order, while the greedy approach (Gauss--Southwell rule) selects at each iteration the coordinate offering the steepest descent. For strongly convex objectives, cyclic CGD's convergence properties are well-established~\cite{wright2015coordinate}, while the greedy variant's behavior has been characterized in~\cite{nutini2015coordinate}. These guarantees extend to general convex functions under cyclic strategies~\cite{saha2013nonasymptotic, beck2013convergence}. In nonconvex optimization, significant attention has focused on saddle point avoidance. Work in~\cites{lee2016gradient, lee2019first} demonstrates that deterministic first-order methods, including cyclic CGD, almost surely escape strict saddle points. The Kurdyka--Łojasiewicz (KŁ) framework~\cite{attouch2010proximal, boct2023inertial} further provides convergence guarantees for coordinate methods in nonconvex and nonsmooth settings. Extensions to variance-reduced or manifold-constrained methods are studied in~\cite{cai2023cyclic, peng2023block}.

Randomized coordinate gradient descent (RCGD) employs various sampling strategies, including uniform/non-uniform random selection and random-permutation approaches~\cites{lee2019random, wright2020analyzing}. For convex optimization, \cite{nesterov2012efficiency} establishes sublinear convergence to the minimum in expectation ($\mathbb{E}(f(x_t))$) for general convex functions, with linear convergence under strong convexity. Further convex convergence results appear in~\cites{liu2014asynchronous, liu2015asynchronous, wright2015coordinate}. The nonconvex case presents distinct challenges: unlike randomized gradient descent that escapes strict saddle points through additive Gaussian noise~\cites{ge2015escaping, jin2017escape, jin2018accelerated}, RCGD's coordinate-wise stochasticity inherently limits this capability. While introducing additive perturbations might help, such modifications risk compromising the algorithm's coordinate structure, necessitating alternative analytical approaches. Relevant nonconvex convergence analyses for RCGD can be found in~\cites{chen2023global, li2019coordinatewise}.

A fruitful research direction interprets iterative algorithms as discrete-time dynamical systems~\cites{lee2016gradient,lee2019first, o2019behavior}. In deterministic settings, the center-stable manifold theorem~\cite{shub2013global} provides a powerful tool for analyzing system behavior near critical points, particularly through the lens of the invariant manifold theorem. This perspective naturally extends to randomized algorithms via random dynamical systems theory. Recent work by Liu and Yuan~\cite{liu2024almost} demonstrates this approach, applying the invariant manifold theorem to prove saddle point avoidance for various stochastic methods. Chen et al.~\cite{chen2023global} developed quantitative finite-block analyses of random dynamical systems to establish convergence properties near saddle points.

The theory of random dynamical systems offers powerful analytical tools, including random stable, unstable, and center manifolds~\cites{arnold1998random, ruelle1979ergodic, ruelle1982characteristic, boxler1989stochastic, guo2016smooth, li2005sternberg}. These constructs capture sample-dependent geometric structures in the state space and, under generic conditions, typically exhibit low dimensionality. This characteristic makes them particularly well-suited for analyzing convergence probabilities to strict saddle points.

The work most closely related to ours is~\cite{chen2023global}, from which our approach differs in two key aspects. First, we eliminate a technical assumption on the objective function while adopting a more practical fixed-stepsize scheme. Second, whereas~\cite{chen2023global} employs linearized finite-block analysis, our work directly applies invariant manifold theory to establish convergence properties - an approach that yields both stronger theoretical guarantees and greater robustness in the analysis.

\subsection{Main results}

We first set up the notations before stating our main theorem. For each $t
\in \bN$, denote  $(\Omega_t,\Sigma_t,\bP_t)$ the usual probability space
for the distribution $\mathcal{U}_{\{1,2,\dots,n\}}$, where
$\mathcal{U}_{\{1,2,\dots,n\}}$ are the uniform distributions on the set
$\{1,2,\dots,n\}$, which is associated to the $t$-th iteration of
Algorithm~\ref{alg:RCGD}. Let $(\Omega,\mathcal{F},\bP)$ be the
product probability space of all $(\Omega_t,\Sigma_t,\bP_t),\ t\in\bN$,
and a sample $\omega\in\Omega$ can be represented as a sequence
$(i_0,i_1,i_2,\dots)$, where $i_t\in \Omega_t$. It is clear that the
iterates generated by Algorithm~\ref{alg:RCGD} is sample-dependent,
and we would use the notation $x_t = x_t(\omega)$ to clarify the
dependence if necessary. 

Moreover, the dynamics of Algorithm~\ref{alg:RCGD} depend on the
initialization $x_0\in \bR^d$ and the sample $\omega\in \Omega$. We set
$\Theta = \bR^d\times \Omega$ that is equipped with a product measure 
\begin{equation*}
    \mu = \textsc{Leb} \times \bP,
\end{equation*}
where $\textsc{Leb}$ is the Lebesgue measure on $\bR^d$.

The primary objects of interest in our analysis are strict saddle points,
and the collection of all strict saddle points is defined as
\begin{equation*}
    \crits := \left\{x\in\bR^d:\nabla f(x)=0,\
    \lambda_{\min}(\nabla^2 f(x))<0 \right\},
\end{equation*}
where $\lambda_{\min}(\nabla^2 f(x))$ is the smallest eigenvalue of the
Hessian matrix $\nabla^2 f(x)$ and we use the subscript $s$ to emphasize
that it is strict. 

Our main results are based on some standard and generic assumptions on the
objective function $f$. The first assumption is that $f$ is two times
continuously differentiable with a bounded and locally Lipschitz
continuous Hessian.

\begin{mainassumption}\label{Assump:Hess_bdd}
    The function $f \in \mathcal{C}^2(\bR^d)$ and the Hessian $\nabla^2 f$
    is uniformly bounded, i.e., there exists $M>0$ such that
    $\norm{\nabla^2 f(x)}\leq M$ for all $x\in\bR^d$. Moreover, for every
    $x^*\in \crits$, there exists a neighborhood $N(x^*)$ of $x^*$ where $\nabla^2 f$ is locally Lipschitz continuous, i.e., $\norm{\nabla^2 f(x) - \nabla^2
    f(y)} \leq L\norm{x-y}$ for all $x, y\in N(x^*)$, where $L = L(x^*) > 0$ is a constant depending on $x^*$.
\end{mainassumption}

Throughout this paper, $\|\cdot\|$ always represents the $\ell_2$-norm for
vectors or its induced matrix norm. Our next assumption is on the
non-degeneracy of the strict saddle points.

\begin{mainassumption}\label{Assump:non_degenerate}
    For every $x^*\in \crits$, $\nabla^2 f(x^*)$ is non-degenerate, i.e.,
    $x^*$ is a non-degenerate critical point of $f$ in the sense that any
    eigenvalue of $\nabla^2 f(x^*)$ is nonzero.
\end{mainassumption}
    
The main theorem of this paper states that the set of all $(x_0,\omega)$
with $x_t(\omega)$ converging to a strict saddle point in $\crits$ is of
measure zero. In particular, consider
\begin{equation*}
    \Theta(x^*) = \left\{(x_0,\omega)\in\Theta : \lim_{t\to + \infty}
    x_t(\omega) = x^*\right\},
\end{equation*}
for each $x^*\in \crits$, and
\begin{equation*}
    \Theta(\crits) = \bigcup_{x^*\in \crits} \Theta(x^*).
\end{equation*}

\begin{maintheorem}\label{thm:main}
    Suppose that Assumptions~\ref{Assump:Hess_bdd} and
    ~\ref{Assump:non_degenerate} hold and that step $0 < \alpha < 1/M$. It
    holds that
    \begin{equation*}
        \mu(\Theta(\crits)) = 0.
    \end{equation*}
\end{maintheorem}

Moreover, with some additional but still standard assumptions, we prove
that $x_t(\omega)$ always converges to a critical point without negative
Hessian eigenvalues, unless $(x_0,\omega)$ is located in a $\mu$-null set.
More specifically, denote by
\begin{equation*}
    \crit :=\{x\in\bR^d:\nabla f(x)=0\},
\end{equation*}
the set of all critical points of $f$, and we have the following corollary.

\begin{maincorollary}\label{cor:global_conv}
    Suppose that Assumptions~\ref{Assump:Hess_bdd} and
    ~\ref{Assump:non_degenerate} hold and that stepsize $\alpha$ satisfies
    $0 < \alpha < 1/M$. Suppose in addition that every $x^*\in\crit$ is an
    isolated critical point and that the level set $L_c(f):=
    \{x\in\bR^d:f(x)\leq c\}$ is bounded for all $c\in\bR$. Then there
    exists $\widehat{\Theta}\subseteq\Theta$ with
    $\mu(\Theta\setminus\widehat{\Theta}) = 0$, such that
    $\{x_t(\omega)\}_{t\in\bN}$ is convergent with the limit in
    $\crit\setminus\crits$ for any $(x_0,\omega)\in\widehat{\Theta}$.
\end{maincorollary}

\subsection{Organization}

The rest of this paper is devoted to the proofs of the results stated
above and is organized as follows. In Section~\ref{sec:RDS}, we present
some preliminaries on random dynamical systems, which are the foundations
of our analysis. Section~\ref{sec:proof_sketch} outlines the proof with a
comparison to~\cite{chen2023global}. All technical lemmas and propositions
are deferred to Section~\ref{sec:technical_lems}.

\section{Preliminaries on random dynamical systems}\label{sec:RDS}

The dynamics of Algorithm~\ref{alg:RCGD} can be rigorously characterized using the notion of random dynamical systems. In particular, given the initilization $x_0$, the trajectory $\{x_t\}_{t\in \bN}$ is fully determined by a random sample $\omega$.  Therefore, analytical tools developed for random dynamical systems would be useful for analyzing the behavior of Algorithm~\ref{alg:RCGD}. This section briefly reviews some fundamental results in random dynamical systems. For a more detailed introduction, we refer the readers to ~\cite{arnold1998random, li2005sternberg, guo2016smooth} and references therein.

\subsection{Definition of random dynamical system}
Consider a probability space $(\Omega, \mathcal{F}, \mathbb{P})$ and let $\mathbb{T}$ denote a semigroup equipped with its Borel $\sigma$-algebra $\mathcal{B}(\mathbb{T})$, playing the role of time, with the convention that $0 \in \mathbb{T}$. Here, $(\Omega, \mathcal{F}, \mathbb{P})$ represents a general probability space, not necessary the one associated to Algorithm~\ref{alg:RCGD}.
Common choices for $\mathbb{T}$ include $\bN$, $\mathbb{Z}$, $\mathbb{R}_{\geq 0}$, and $\mathbb{R}$. The random dynamical system is defined as follows.

\begin{definition}[Metric dynamical system] A metric dynamical system on a probability space $(\Omega,\mathcal{F},\bP)$ is a family of maps $\{\theta(t):\Omega\rightarrow\Omega\}_{t\in\mathbb{T}}$ satisfying that
    \begin{itemize}
        \item[(i)] The mapping $\mathbb{T}\times\Omega\rightarrow\Omega,\ (t,\omega)\mapsto\theta(t)\omega$ is measurable;
        
        \item[(ii)] It holds that $\theta(0)=\mathrm{Id}_{\Omega}$ and $\theta(t+s)=\theta(t)\circ\theta(s),\ \forall\ s,t\in\mathbb{T}$;
        
        \item[(iii)] $\theta(t)$ is $\bP$-preserving for any $t\in\mathbb{T}$, where we say a map $\theta:\Omega\rightarrow\Omega$ is $\bP$-preserving if 
        \begin{equation*}
            \bP(\theta^{-1}B)=\bP(B),\quad \forall\ B\in\mathcal{F}.
        \end{equation*}
    \end{itemize}
\end{definition}

\begin{definition}[Random dynamical system]
    Let $(X,\mathcal{F}_X)$ be a measurable space and let $\{\theta(t):\Omega\rightarrow\Omega\}_{t\in\mathbb{T}}$ be a metric dynamical system on $(\Omega,\mathcal{F},\bP)$. Then a random dynamical system on $(X,\mathcal{F}_X)$ 
    over $\{\theta(t)\}_{t\in\mathbb{T}}$ is a measurable map 
    \begin{equation*}
        \begin{split}
            \varphi:\mathbb{T}\times\Omega\times X&\rightarrow\quad X,\\
            (t,\omega,x)\ \  & \mapsto \varphi(t,\omega,x),
        \end{split}
    \end{equation*}
    satisfying the following cocycle property: for any $\omega\in\Omega$, $x\in X$, and $s,t\in \mathbb{T}$, it holds that 
    \begin{equation*}
        \varphi(0,\omega,x)=x,
    \end{equation*}
    and that
    \begin{equation}\label{cocycle}
        \varphi(t+s,\omega,x)=\varphi(t,\theta(s)\omega,\varphi(s,\omega,x)).
    \end{equation}
\end{definition}
The cocycle property \eqref{cocycle} plays a fundamental role in the theory of random dynamical systems. Intuitively, it means that if the system evolves for time $s$ and reaches the state $x_s = \varphi(s,\omega,x)$, then continuing the evolution is equivalent to restarting the system from $x_s$ with a shifted random sample $\theta(s)\omega$. In other words, the metric dynamical system $\theta(s)$ maps $\omega$ to another sample controlling the dynamics starting at time $s$. The map $\varphi(t, \omega, \cdot)$ acts on the space $X$; for convenience and with a slight abuse of notation, we also denote this map by $\varphi(t, \omega)$, and write $\varphi(t, \omega)x$ in place of $\varphi(t, \omega, x)$. Under this notation, the cocycle property \eqref{cocycle} takes the form:
\begin{equation*}
    \varphi(t + s, \omega) = \varphi(t, \theta(s)\omega) \circ \varphi(s, \omega).
\end{equation*}

In this work, we restrict our attention to the two-sided discrete-time case $\mathbb{T} = \mathbb{Z}$, which corresponds to the dynamics of Algorithm~\ref{alg:RCGD} and its inverse system. The inversion is included due to a technical reason stated later. We only consider the metric dynamical system $\theta(t) = \theta^t$, with $\theta:\Omega\to\Omega$ being a $\mathbb{P}$-preserving map and $\theta^t$ denotes the $t$-fold composition of $\theta$. We also assume that the state space is $X = \mathbb{R}^d$ throughout this paper.

\subsection{Multiplicative ergodic theorem}
Let $A: \Omega \to \mathrm{GL}(d, \mathbb{R})$ be a measurable map. We consider a linear random dynamical system defined by
\begin{equation*}
    x_t = \Phi(t, \omega)x_0,
\end{equation*}
where the $\Phi(t, \omega)$ is a product of random matrices
\begin{equation*}
    \Phi(t, \omega) = \begin{cases}
        A(\theta^{t-1}\omega) \cdots A(\theta\omega) A(\omega), &\text{if } t>0, \\
        I, & \text{if }t =0,\\
        A(\theta^{t}\omega)^{-1} \cdots A(\theta^{-2}\omega)^{-1} A(\theta^{-1}\omega)^{-1}, &\text{if } t<0.
    \end{cases} 
\end{equation*}
In this setting, the evolution of the system is well characterized by the celebrated multiplicative ergodic theorem (also known as Oseledets' theorem), which we state as Theorem~\ref{thm:MET}. We use the notation $\Phi$ for this linear system, reserving $\varphi$ for the nonlinear dynamics.

\begin{theorem}[Multiplicative ergodic theorem, {~\cite[Theorem 3.4.11]{arnold1998random}}]
    \label{thm:MET}Suppose that 
    \begin{equation}\label{pre:MET}
        (\log\|A(\cdot)\|)_+,\ (\log\|A(\cdot)^{-1}\|)_+\in L^1(\Omega,\mathcal{F},\bP),
    \end{equation}
    where we have used the short-hand $a_+:=\max\{a,0\}$. Then there exists an $\theta$-invariant $\widetilde{\Omega}\in \mathcal{F}$ with $\bP(\widetilde{\Omega})=1$, such that the followings hold for any $\omega\in\widetilde{\Omega}$:
    \begin{itemize}
        \item[(i)] It holds that the limit
        \begin{equation}\label{eq:lim_Lambda_omega}
        \Lambda(\omega)=\lim_{t\rightarrow + \infty}\left(\Phi(t,\omega)^{\top}\Phi(t,\omega)\right)^{1/2t}
        \end{equation}
        exists and is a positive definite matrix. Here $\Phi(t, \omega)^{\top}$ denotes the transposition of the matrix (as $\Phi(t, \omega)$ is a linear map on $X$).
        \item[(ii)] Suppose  $\Lambda(\omega)$ has $p(\omega)$ distinct eigenvalues, which are ordered as $e^{\lambda_1(\omega)}>e^{\lambda_2(\omega)}>\cdots> e^{\lambda_{p(\omega)}(\omega)}>0$, and let $V_i(\omega)$ be the eigenspace associated with $e^{\lambda_i(\omega)}$ with dimension $d_i(\omega)$, for $i=1,2,\dots,p(\omega)$. Then the functions $p(\cdot)$, $\lambda_i(\cdot)$, and $d_i(\cdot)$, $i=1,2,\dots,p(\cdot)$, are all measurable and  $\theta$-invariant on $\widetilde{\Omega}$.
        \item[(iii)] There exists a splitting
        \begin{equation}\label{eq:Oseledets_split}
            \bR^d = E_1(\omega) \oplus E_2(\omega) \oplus \cdots\oplus E_{p(\omega)}(\omega)
        \end{equation}
        of $\bR^d$ into random subspaces $E_i(\omega)$ with $\mathrm{dim} \, E_i(\omega) = d_i(\omega)$, such that
        \begin{equation}\label{conv:Lya_exp}
            \lim_{t\to\pm\infty} \frac{1}{t} \log\norm{\Phi(t,\omega) x_0} = \lambda_i(\omega) ~~\Longleftrightarrow ~~ x_0\in E_i(\omega)\setminus \{0\},
        \end{equation}
        and
        \begin{equation*}
            A(\omega) E_i(\omega) = E_i(\theta\omega).
        \end{equation*}
\item[(iv)] When $(\Omega,\mathcal{F},\bP,\theta)$ is ergodic, i.e., every $B\in\mathcal{F}$ with $\theta^{-1}B=B$ satisfies $\bP(B)=0$ or $\bP(B)=1$, the functions  
$p(\cdot)$, $\lambda_i(\cdot)$, and $d_i(\cdot)$, $i=1,2,\dots,p(\cdot)$, are constant on $\widetilde{\Omega}$.
    \end{itemize}
\end{theorem}

In the rest of Section~\ref{sec:RDS}, we slightly abuse the notation by denoting $\widetilde{\Omega}$ by $\Omega$ for simplicity. In other words, we assume that all statements in Theorem~\ref{thm:MET} are true for every $\omega$, with the null set $\Omega\setminus\widetilde{\Omega}$ being removed in advance.

In Theorem~\ref{thm:MET}, $\lambda_i(\omega),\ i=1,2,\dots,p_i(\omega)$ are the so-called Lyapunov exponents, $E_i(\omega),\ i=1,2,\dots,p_i(\omega)$ are named as the Oseledets subspaces, and the splitting \eqref{eq:Oseledets_split} is termed Oseledets splitting. The limiting behavior of $x_t(\omega)$ is characterized precisely by \eqref{conv:Lya_exp}. In particular, $\norm{x_t(\omega)}$ grows exponentially as $t\to+\infty$ if $x_0$ has nontrivial projection onto at least one $E_i(\omega)$ with $\lambda_i(\omega)>0$, along the spliting \eqref{eq:Oseledets_split}. Otherwise, $\norm{x_t(\omega)}$ either decays exponentially or has sub-exponential behavior as $t\to+\infty$. This observation motivates the following decomposition:
\begin{equation*}
    \bR^d = E_u(\omega) \oplus E_{cs}(\omega),
\end{equation*}
where
\begin{equation*}
    E_u(\omega) = \bigoplus_{\lambda_i>0} E_i(\omega), \quad\text{and}\quad E_{cs}(\omega) = \bigoplus_{\lambda_i\leq 0} E_i(\omega).
\end{equation*}
We call $E_u(\omega)$ the unstable Oseledets subspace and $E_{cs}(\omega)$ the center-stable Oseledets subspace. It can be seen that a necessary condition that $x_t(\omega)$ converges to $0$, or even stays bounded, as $t\to+\infty$ is that $x_0(\omega)\in E_{cs}(\omega)$. Though we mainly focus on $t\to+\infty$ as it is of particular interest in the context of Algorithm~\ref{alg:RCGD}, we remark that similar limit behavior is also true for $t\to-\infty$ since \eqref{conv:Lya_exp} states a two-sided limit.

\subsection{Center-stable manifold theorem} For a nonlinear random dynamical system $\varphi(t,\omega,x)$, we say that $x^*$ is an equilibrium if
\begin{equation*}
    \varphi(t,\omega,x^*) = x^*,\quad\forall~t\in \bZ,\ \omega\in\Omega.
\end{equation*}
The behavior of the system near $x^*$ can be approximated by its linearization at $x^*$. More specifically, there is a local manifold, termed the center-stable manifold, that is tangent to the center-stable Oseledets subspace associated with the linearized system, such that a necessary condition that $x_t$ converges to $x^*$ is that $x_0$ lies on the center-stable manifold. To state the theorem rigorously, we write
\begin{equation}\label{eq:linearization}
    \varphi(t,\omega,x) = \Phi(t,\omega)x + F(t,\omega,x),
\end{equation}
where the equilibrium is assumed to be $x^* = 0$ without loss of generality and $\Phi(t,w)$ is the linearized system given by $\Phi(t,w) = D_x \varphi(t,\omega,0)$, and make the following assumption.

\begin{definition}[Tempered random variable]
    We say that a random variable $R: \Omega \to (0,\infty)$ is tempered from above if
    \begin{equation*}
    \lim_{t \to \pm \infty} \frac{1}{t} \left( \log R(\theta^t \omega)\right)_{+} = 0
    \end{equation*}
    holds almost surely, and a random variable $R: \Omega \to (0,\infty)$ is called tempered from below if $1/R$ is tempered from above. Moreover, a random variable is called tempered if it is both tempered from above and tempered from below.
\end{definition}

\begin{assumption}\label{asp:SMT}
    We assume that the random dynamical system $\varphi(t,\omega,x)$ in \eqref{eq:linearization} with an equilibrium $x^*=0$ satisfies the following conditions:
    \begin{itemize}
        \item[(i)] $\varphi(t,\omega,x)$ is a $\mathcal{C}^1$ random dynamical system, and $\Phi (t,\omega) = D_x\varphi(t,\omega,0)$ satisfies the condition~\eqref{pre:MET}. 
        \item[(ii)] There is a ball $U (\omega) = \{x \in \bR^d: \norm{x} <\rho_0(\omega)\}$ where $\rho_0 : \Omega \to (0, \infty)$ is tempered from below such that
        \begin{equation}\label{eq:asp_SMT}
            \begin{split}
                &\sup_{x\in U (\omega)} \norm{D_x^k F (1,\omega,x)}\leq B_k(\omega),\quad\forall~ 0 \leq k \leq 1 ,\ \omega \in \Omega, \\
                &\norm{D_x F (1,\omega,x)}\leq B(\omega) \norm{x}, \quad\forall~ x\in U (\omega),\ \omega \in \Omega,
            \end{split}
        \end{equation}
        where $B, B_k:\Omega\to(0,+\infty)$ is tempered from above for $k=0,1$.
    \end{itemize}
\end{assumption}

\begin{remark}
    Assumption~\ref{asp:SMT} presents a minor modification of the conditions employed in~\cite{guo2016smooth, li2005sternberg} according to their proofs. This assumption aligns with \cite[Lemma 7.5.11]{arnold1998random}.
\end{remark}

Now we can state the center-stable manifold theorem.

\begin{theorem}[Center-stable manifold theorem ~\cite{arnold1998random, guo2016smooth, li2005sternberg}] \label{thm:SMT}
Let $\varphi(t,\omega,x)$ be a random dynamical system as in \eqref{eq:linearization} with an equilibrium $x^*=0$ and satisfy Assumption~\ref{asp:SMT}. Then there exist a tempered variable $\rho:\Omega\to(0,+\infty)$ and a measurable function
    \begin{equation*}
        h^{cs} : E_{cs}(\omega) \times \Omega \to E_u(\omega),
    \end{equation*}
    such that the following are satisfied:
    \begin{itemize}
        \item[(i)]  $h^{cs}(y,\omega)$ is measurable in $(y,\omega)$ and is $\mathcal{C}^1$ in $y$ with
        \begin{equation*}
            \mathrm{Lip} \, h^{cs}(\cdot, \omega)<1,\ h^{cs}(0,\omega)=0,\ \text{and}\ D_y h(0,\omega)=0.
        \end{equation*}
        \item[(ii)] For any $x\in \bR^d$ and any $\omega\in \Omega$, if
        \begin{equation*}
            \varphi(t,\omega,x) \to 0,\quad\text{as }t\to+\infty,
        \end{equation*}
        and
        \begin{equation}\label{eq:local_sequence}
            \|\varphi(t,\omega,x)\| < \rho(\theta^t \omega),\quad\forall~t\in\bN,
        \end{equation}
        then
        \begin{equation*}
            x\in W^{cs}(\omega) = \left\{ y + h^{cs}(y,\omega) : y\in E_{cs}(\omega) \right\}.
        \end{equation*}
    \end{itemize}
\end{theorem}

The center-stable manifold theorem plays a fundamental role in our analysis. There are richer and stronger results on and related to center-stable manifolds in ~\cite{li2005sternberg, guo2016smooth}, and in Theorem~\ref{thm:SMT}, we only state the results that are necessary in our proofs. In particular, Theorem~\ref{thm:SMT} states that there exists a random manifold $W^{cs}(\omega)$, tangent to the center-stable Oseledets subspace $E_{cs}(\omega)$ and defined as the graph of $h^{cs}(\omega)$, with $\textrm{dim}\, W^{cs}(\omega) = \textrm{dim}\, E^{cs}(\omega)$, such that a trajectory converging to $x^*=0$ must start at a point in $W^{cs}(\omega)$, provided that it always stays in a tempered ball in the sense of \eqref{eq:local_sequence}.

\section{Proof sketch and discussion}\label{sec:proof_sketch}

We present the main proof outlines and ideas in this section, with all technical lemmas and propositions deferred to Section~\ref{sec:technical_lems}. We also make some technical discussion and comparisons with ~\cite{chen2023global}.

\subsection{Setup of the random dynamical system}
We rigorously define the random dynamical system associated with Algorithm~\ref{alg:RCGD}. In this setup, the system is reversible and has two-sided discrete time $t\in\bZ$, which is because we cannot find a standard version of the center-stable manifold theorem for one-sided time $t\in\bN$ in the existing literature.
\begin{itemize}[wide]
    \item{\emph{Probability space}.} For each $t \in \bZ$, denote  $(\Omega_t,\Sigma_t,\bP_t)$ the usual probability space for the distribution $\mathcal{U}_{\{1,2,\dots,n\}}$, where $\mathcal{U}_{\{1,2,\dots,n\}}$ are the uniform distributions on the set $\{1,2,\dots,n\}$. As mentioned in the introduction, Algorithm~\ref{alg:RCGD} is naturally equipped with $(\Omega,\mathcal{F},\bP)$ that is the product space of $(\Omega_t,\Sigma_t,\bP_t)$ for $t\in\bN$. We extended the probability by including the negative times, i.e., we consider $(\Omega^e,\mathcal{F}^e,\bP^e)$ that is the product space of $(\Omega_t,\Sigma_t,\bP_t)$ for $t\in\bZ$. It is clear that for every $A\in\mathcal{F}$,
    \begin{equation}\label{eq:extended_prob}
        \bP(A) = \bP^e\left(\left(\prod_{t\in\bZ_{<0}} \Omega_t \right) \times A \right).
    \end{equation}
    It can be seen that every sample in $\Omega^e$ is a sequence of indices $\omega^e = (\cdots, i_{-2},i_{-1},i_0,i_1,i_2,\cdots)$ and we denote $\pi_t$ as the projection map from $(\Omega^e,\mathcal{F}^e,\bP^e)$ to $(\Omega_t,\Sigma_t,\bP_t)$ that sends $\omega$ to $i_t$, for every $t\in\bZ$.
    
    \item{\emph{Metric dynamical system}.}
    The metric dynamical system on $\Omega^e$ is $\theta(t) = \theta^t$ for $t\in\bZ$, where $\theta:\Omega^e\to\Omega^e$ is the shifting operator defined via
    \begin{equation*}
        \pi_t(\theta\omega^e) = \pi_{t+1}(\omega^e),\quad \forall~\omega^e\in\Omega^e,\ t\in\bZ,
    \end{equation*}
    which essentially shifts every element in $\omega^e$ leftward for one position.
    It is clear that $\theta$ is measurable and $\bP^e$-preserving.
    
    \item{\emph{Random dynamical system}.}
    For any $\omega^e\in\Omega^e$, we define a (nonlinear) map on $\bR^d$ via 
    \begin{equation*}
        \begin{split}
            \phi(\omega^e):\bR^d&\rightarrow\quad\quad\quad \bR^d,\\
            x\ &\mapsto x-\alpha e_i e_i^{\top} \nabla f(x),
        \end{split}
    \end{equation*}
    where $i=\pi_0(\omega^e)$ is the index in $\omega^e$ at the time $t=0$. It can be seen that $\phi(\omega^e)$ implements one iteration of Algorithm~\ref{alg:RCGD} with the sampled index being $i=\pi_0(\omega^e)$. Then we define $\varphi(t,\omega^e):\bR^d\to\bR^d$ via 
    \begin{equation*}
        \varphi(t, \omega^e) = \begin{cases}
            \phi(\theta^{t-1}\omega^e) \circ \cdots \circ \phi(\theta\omega^e)\circ \phi(\omega^e), &\text{if } t>0, \\
            Id, & \text{if }t =0,\\
            \phi(\theta^{t}\omega^e)^{-1}\circ \cdots\circ \phi(\theta^{-2}\omega^e)^{-1}\circ \phi(\theta^{-1}\omega^e)^{-1}, &\text{if } t<0,
        \end{cases} 
    \end{equation*}
    which satisfies the cocycle property \eqref{cocycle}. Here, $\phi(\omega^e):\bR^d\to\bR^d$ is bijective and is hence invertible if $\alpha<1/M$, where $M$ is the constant as in Assumption~\ref{Assump:Hess_bdd}. We will rigorously characterize the invertibility of $\phi(\omega^e)$ in Section~\ref{sec:invertible_phi}. Thus, $\varphi(t,\omega^e)$ defines a random dynamical system on $X=\bR^d$ over $\{\theta^t\}_{t\in\bZ}$, and one can see that $\varphi(t,\omega^e)$ essentially implements the first $t$ iterations of Algorithm~\ref{alg:RCGD} for $t\geq 0$.
\end{itemize}

\subsection{Restatement of Theorem~\ref{thm:main}} One can restate Theorem~\ref{thm:main} with the extended probability space $(\Omega^e,\mathcal{F}^e,\bP^e)$. In particular, denote $\Theta^e = \bR^d\times \Omega^e$ that is equipped with a product measure 
\begin{equation*}
    \mu^e = \textsc{Leb} \times \bP^e,
\end{equation*}
and then define
\begin{equation*}
    \Theta^e(x^*) = \left\{(x_0,\omega^e)\in\Theta^e : \lim_{t\to + \infty}\varphi(t,\omega^e,x_0) = x^*\right\},
\end{equation*}
for each $x^*\in\crits$, and
\begin{equation*}
    \Theta^e(\crits) = \bigcup_{x^*\in \crits} \Theta^e(x^*).
\end{equation*}
One can thus see from \eqref{eq:extended_prob} that Theorem~\ref{thm:main} can be equivalently restated as follows.

\begin{theorem}[Restatement of Theorem~\ref{thm:main}]\label{thm:main_restate}
    Suppose that Assumptions~\ref{Assump:Hess_bdd} and ~\ref{Assump:non_degenerate} hold and that $0 < \alpha < 1/M$. It holds that
    \begin{equation*}
        \mu^e(\Theta^e(\crits)) = 0.
    \end{equation*}
\end{theorem}

Theorem~\ref{thm:main_restate} essentially states that for $\mu^e$-almost surely $(x_0,\omega^e)$ will not converge to a strict saddle point $x^*\in\crits$. To prove Theorem~\ref{thm:main_restate}, one only needs to investigate the measure of $\Theta^e(x^*)$ for every $x^*\in\crits$, which is stated in the following theorem.

\begin{theorem}\label{thm:single_xstar}
    Suppose that Assumptions~\ref{Assump:Hess_bdd} holds and that $0 < \alpha < 1/M$. For any $x^*\in\crits$, if $x^*$ is a non-degenerate critical point of $f$, i.e., all eigenvalues of $\nabla^2 f(x^*)$ are nonzero, then
    \begin{equation*}
        \mu^e(\Theta^e(x^*)) = 0.
    \end{equation*}
\end{theorem}

In particular, the proof of Theorem~\ref{thm:main_restate} is straightforward based on Theorem~\ref{thm:single_xstar}.

\begin{proof}[Proof of Theorem~\ref{thm:main_restate}]
    It follows from Assumption~\ref{Assump:non_degenerate} that every $x^*\in \crits$ is an isolated critical point of $f$, which is because that $\nabla f(x)=\nabla^2 f(x^*)(x-x^*)+o(\norm{x-x^*}) \neq 0$ for any $x\neq x^*$ in a small neighborhood of $x^*$. This implies that $\crits$ is countable. Then according to Theorem~\ref{thm:single_xstar}, we can obtain that
    \begin{equation*}
        \mu^e(\Theta^e(\crits)) = \sum_{x^*\in\crits} \mu^e(\Theta^e(x^*)) = 0,
    \end{equation*}
    which completes the proof of Theorem~\ref{thm:main_restate}.
\end{proof}

Corollary~\ref{cor:global_conv} is also an immediate consequence.

\begin{proof}[Proof of Corollary~\ref{cor:global_conv}]
    Using the same arguments as in ~\cite[Proofs of Proposition 4.11 and Proposition 4.12]{guo2016smooth}, it can be shown that $x_t(\omega)$ converges to a point in $\crit$ as $t\to+\infty$ unless $(x_0,\omega)$ is in a $\mu$-null set. One can further exclude $\Theta(\crits)$, that is also a $\mu$-null set as in Theorem~\ref{thm:main}, which guarantees that the limit is in $\crit\setminus\crits$.
\end{proof}

\subsection{Proof of Theorem~\ref{thm:single_xstar}}
In this subsection, we present the main proof outline of Theorem~\ref{thm:single_xstar}, with proofs of some technical lemmas and propositions deferred to Section~\ref{sec:technical_lems}. We consider any $x^*\in\crits$ and assume that $x^*=0$ without loss of generality. The linearization of $\varphi(t,\omega^e)$ at $x^*=0$ is
\begin{equation*}
\Phi^H(t, \omega^e) = \begin{cases}
    A^H(\theta^{t-1}\omega^e) \cdots A^H(\theta\omega^e) A^H(\omega^e), &\text{if } t>0, \\
    I, & \text{if }t =0,\\
    A^H(\theta^{t}\omega^e)^{-1} \cdots A^H(\theta^{-2}\omega^e)^{-1} A^H(\theta^{-1}\omega^e)^{-1}, &\text{if } t<0,
\end{cases} 
\end{equation*}
where 
\begin{equation*}
    A^H(\omega^e) = I-\alpha e_i e_i^\top H,\quad H = \nabla^2 f(x^*).
\end{equation*}

A crucial step in the proof of Theorem~\ref{thm:single_xstar} is to apply the center-stable manifold theorem, i.e., Theorem~\ref{thm:SMT}, for which one needs to validate Assumption~\ref{asp:SMT}. We include the detailed validation with Assumption~\ref{Assump:Hess_bdd} and $\alpha < 1/M$ in Section~\ref{sec:validate_asp_SMT}. Moreover, we need the following two propositions.

\begin{proposition}\label{prop:positive_Lyapunov}
    Let $H=\nabla f(x^*)$ have a negative eigenvalue and $0 <\alpha < 1/\max_{1\leq i\leq d}|H_{ii}|$, then the largest Lyapunov exponent of $\Phi^H(t,\omega^e)$ is positive.
\end{proposition}

\begin{proposition}\label{prop:converexpo}
    Suppose that Assumption~\ref{Assump:Hess_bdd} holds and that $0<\alpha<1/M$. For any $x^*\in\crits$, if $x^*$ is a non-degenerate critical point of $f$, i.e., all eigenvalues of $\nabla^2 f(x^*)$ are nonzero, then there exists $\widetilde{\Theta} \subseteq \Theta = \bR^d\times \Omega $ with $\mu(\Theta \setminus \widetilde{\Theta}) = 0$, such that for any $(x_0,\omega)\in \widetilde{\Theta}$, if $x_t(\omega) \to x^*$ as $t\to+\infty$, then $x_t(\omega) \to x^*$ exponentially as $t\to+\infty$.
\end{proposition}

Throughout this paper, we say that a sequences $y_t$ in $\bR^d$ or $\bR$ converges exponentially to $y^*$ as $t\to+\infty$ if
\begin{equation*}
    \limsup_{t\to+\infty} \frac{1}{t} \log \norm{y_t-y^*} < 0.
\end{equation*}
Note that our notion of exponential convergence is essentially at least exponential convergence, which includes convergence rates faster than exponential rates.
The proof of Proposition~\ref{prop:positive_Lyapunov} follows exactly the same line as in the proof of ~\cite{chen2023global}*{Proposition 3.1} and is hence omitted. The proof of Proposition~\ref{prop:converexpo} is presented in Section~\ref{sec:pf_converexpo}. Proposition~\ref{prop:converexpo} states that almost every convergent trajectory has an exponential convergence rate, which eventually makes the condition \eqref{eq:local_sequence} true and hence guarantees that we can identify convergent trajectories with points on the center-stable manifold. In addition, Proposition~\ref{prop:positive_Lyapunov} shows that the center-stable manifold is of dimension at most $d-1$ due to the presence of the positive Lyapunov exponent, and is hence of Lebesgue measure zero. These observations lead to Theorem~\ref{thm:single_xstar}. We then present the proof of Theorem~\ref{thm:single_xstar} based on Assumption~\ref{asp:SMT}, Proposition~\ref{prop:positive_Lyapunov}, and Proposition~\ref{prop:converexpo}.

\begin{proof}[Proof of Theorem~\ref{thm:single_xstar}]
    Let $\widetilde{\Omega}^e\subset \Omega^e$ be $\theta$-invariant with $\bP^e(\widetilde{\Omega}^e)=1$, so that all statements in Theorem~\ref{thm:MET} and Theorem~\ref{thm:SMT} are true for every $\omega^e\in\widetilde{\Omega}^e$, and let $\rho:\Omega^e\to(0,+\infty)$ be the tempered random variable as in Theorem~\ref{thm:SMT} satisfying
    \begin{equation*}
        \lim_{t\to\pm\infty} \frac{1}{t} \log \rho(\theta^t\omega^e) = 0,\quad\forall~\omega^e\in \widetilde{\Omega}^e.
    \end{equation*}
    Moreover, it follows from the Kolmogorov's zero-one law that $(\Omega^e,\mathcal{F}^e,\bP^e,\theta)$ is ergodic, which implies that $p(\omega^e),\lambda_i(\omega^e),d_i(\omega^e)$ are all constant over $\omega^e\in \widetilde{\Omega}^e$. We thus drop the dependence on $\omega^e$ and denote these constants by $p,\lambda_i,d_i$ for simplicity. Denote $E^{cs}(\omega^e)$ and $W^{cs}(\omega^e)$ as the center-stable Oseledets subspace and the center-stable manifold, respectively. It follows from Proposition~\ref{prop:positive_Lyapunov} that
    \begin{equation}\label{eq:dim_Wcs}
        \mathrm{dim} \, W^{cs}(\omega^e) = \mathrm{dim}\, E^{cs}(\omega^e) = d - \sum_{\lambda_i>0} d_i \leq d-1,\quad\forall~\omega^e\in\widetilde{\Omega}^e,
    \end{equation}

    Define
    \begin{equation*}
        \widetilde{\Theta}^e = \left(\prod_{t\in\bZ_{<0}}\Omega_t\right)\times \widetilde{\Theta}\subseteq\Theta^e = \bR^d\times\Omega^e,
    \end{equation*}
    where $\widetilde{\Theta}\subseteq\Theta$ is from Proposition~\ref{prop:converexpo}. Consider any
    \begin{equation*}
        (x_0,\omega^e)\in \Theta^e(x^*) \cap \widetilde{\Theta}^e \cap \left(\bR^d\times\widetilde{\Omega}^e\right).
    \end{equation*}
    It follows from Proposition~\ref{prop:converexpo} that
    \begin{equation*}
        \varphi(t,\omega^e,x_0) \to x^* = 0,\quad t\to+\infty,
    \end{equation*}
    exponentially, i.e.,
    \begin{equation*}
        \limsup_{t\to+\infty} \frac{1}{t} \log \norm{\varphi(t,\omega^e,x_0)} < 0.
    \end{equation*}
    Therefore, there exists $T(x_0,\omega^e)\in\bN$ such that
    \begin{equation*}
        \frac{1}{t} \log \norm{\varphi(t,\omega^e,x_0)} < \frac{1}{t} \log \rho(\theta^t\omega^e),\quad\forall~t\geq T(x_0,\omega^e),
    \end{equation*}
    and equivalently that,
    \begin{equation*}
        \norm{\varphi(t,\omega^e,x_0)} < \rho(\theta^t\omega^e),\quad\forall~t\geq T(x_0,\omega^e).
    \end{equation*}
    We then have that
    \begin{equation*}
    \varphi\left(t,\theta^{T(x_0,\omega^e)}\omega^e,\varphi\left(T(x_0,\omega^e),\omega^e,x_0\right)\right) \to x^* = 0,\quad t\to +\infty,
    \end{equation*}
    and 
    \begin{equation*}
        \norm{\varphi\left(t,\theta^{T(x_0,\omega^e)}\omega^e,\varphi\left(T(x_0,\omega^e),\omega^e,x_0\right)\right)} < \rho\left(\theta^t \theta^{T(x_0,\omega^e)}\omega^e\right),\quad\forall~t\in\bN.
    \end{equation*}
    Using Theorem~\ref{thm:SMT}, we can conclude that
    \begin{equation*}
        \varphi\left(T(x_0,\omega^e),\omega^e,x_0\right) \in W^{cs}\left(\theta^{T(x_0,\omega^e)}\omega^e\right),
    \end{equation*}
    which implies that
    \begin{equation*}
        x_0\in \varphi(T(x_0,\omega^e),\omega^e)^{-1}\left(W^{cs}\left(\theta^{T(x_0,\omega^e)}\omega^e\right)\right) \subseteq \bigcup_{t\in\bN} \varphi(t,\omega^e)^{-1} \left(W^{cs}(\theta^t \omega^e)\right),
    \end{equation*}
    where the Lebesgue measure of the set on the right-hand side can be computed from \eqref{eq:dim_Wcs} as
    \begin{equation*}
        \textsc{Leb}\left(\bigcup_{t\in\bN} \varphi(t,\omega^e)^{-1} \left(W^{cs}(\theta^t \omega^e)\right)\right) \leq \sum_{t\in\bN} \textsc{Leb}\left(\varphi(t,\omega^e)^{-1} \left(W^{cs}(\theta^t \omega^e)\right)\right) = 0,
    \end{equation*}
    since the image of a manifold of dimension at most $d-1$ under a $\mathcal{C}^1$ map is of Lebesgue measure zero. Therefore, by Fubini's theorem, one obtains that
    \begin{equation*}
        \mu^e\left(\Theta^e(x^*) \cap \widetilde{\Theta}^e \cap \left(\bR^d\times\widetilde{\Omega}^e\right)\right)\leq \int_{\Omega^e} \textsc{Leb}\left(\bigcup_{t\in\bN} \varphi(t,\omega^e)^{-1} \left(W^{cs}(\theta^t \omega^e)\right)\right) \bP^e(\mathrm{d}\omega^e) = 0,
    \end{equation*}
    and hence that
    \begin{equation*}
        \mu^e\left(\Theta^e(x^*)\right) \leq \mu^e\left(\Theta^e(x^*) \cap \widetilde{\Theta}^e \cap \left(\bR^d\times\widetilde{\Omega}^e\right)\right) + \mu^e\left(\Theta^e \setminus \widetilde{\Theta}^e\right) + \mu^e\left(\Theta^e \setminus (\bR^d\times\widetilde{\Omega}^e)\right) = 0,
    \end{equation*}
    where we used $\mu(\Theta \setminus  \widetilde{\Theta}) = 0$ and $\bP^e(\Omega^e\setminus \widetilde{\Omega}^e) = 0$. The proof is completed.
\end{proof}

\subsection{Comparison and discussion} 

As the conclusion of this section, we make a technical comparison with ~\cite{chen2023global}, whose main theorem is of similar style to ours, i.e., for any $x_0$ that is not a strict saddle point, $x_t$ almost surely does not converge to a strict saddle point as $t\to+\infty$. The main difference is that the following assumption is required in ~\cite{chen2023global}.

\begin{assumption}[{~\cite[Assumption 3]{chen2023global}}]\label{Assump:non_zero_proj}
    For every $x^*\in \crits$, it holds that $\mathcal{P}_+^H(\omega) e_i\neq 0$, for every $i\in\{1,2,\dots,d\}$ and almost every $\omega\in\Omega$, where $\mathcal{P}_+^H(\omega)$ is the orthogonal projection onto $W_{+}^H(\omega) = \bigoplus_{\lambda_i(\omega) >0} V_i^H(\omega)$ and $V_i^H(\omega)$ is the eigenspace corresponding to the eigenvalue $\lambda_i(\omega)$ of $\Lambda^H(\omega)$ as in \eqref{eq:lim_Lambda_omega}, $i=1,2,\dots,p(\omega)$, for the linearized system at $x^*$ with $H=\nabla^2 f(x^*)$.
\end{assumption}

We remark that ~\cite{chen2023global} only considers sample $\omega$ corresponding to one-sided time $\mathbb{T} = \bN$, without extending to $\omega^e$ and two-sided time $\mathbb{T} = \bZ$. Therefore, there is no explicit notion of center-stable and unstable Oseledets subspaces, but the limiting matrix $\Lambda^H(\omega)$ and its eigenspaces $V_i^H(\omega),\ i=1,2,\dots,p(\omega)$ are still well-defined, as they are defined one-sidedly for $t\to+\infty$ in \eqref{eq:lim_Lambda_omega}. In our notation, one can verify that $W_+^H(\omega^e) = \bigoplus_{\lambda_i >0} V_i^H(\omega^e)$ is the orthogonal complement of the center-stable Oseledets subspace in $\bR^d$.

The main proof in ~\cite{chen2023global} is that, with Assumption~\ref{Assump:non_zero_proj} and additional randomness in stepsizes, the iterate would have a non-negligible component in $W_{+}^H(\omega)$ with high probability, which will be amplified sufficiently and hence drive the dynamics to leave the neighborhood of $x^*\in\crits$, avoiding the convergence to $x^*$. However, there are two main drawbacks of this analytical framework. Firstly, additional randomness is a bit artificial, making the randomized coordinate gradient descent not in its simplest format. In fact, the sample $\omega$ in ~\cite{chen2023global} is a sequence of not only the random coordinates, but also the random stepsizes. Secondly, though the techinical Assumption~\ref{Assump:non_zero_proj} can be verified generically, it excludes some practical Hessian matrices, such as $H=\nabla^2 f(x^*)=\text{diag}(H_1,H_2)$ where all eigenvalues of $H_1$ are positive and all eigenvalues of $H_2$ are negative, which is already acknowledged in ~\cite{chen2023global}.

Overall, our present work overcomes some technical obstacles in ~\cite{chen2023global}, and establishes a neater and more general analytical framework for the convergence of randomized coordinate gradient descent. This general framework could be potentially applied or extended for other randomized algorithms.

\section{Technical lemmas and propositions}\label{sec:technical_lems}

We collect all technical lemmas and propositions in this section.

\subsection{Invertibility of $\phi(\omega^e)$}
\label{sec:invertible_phi}

Recall that $\phi(\omega^e):x \mapsto x - \alpha e_i e_i^\top \nabla f(x)$, where $i = \pi_0(\omega^e)$ is the index in $\omega^e$ at the time $t=0$, and we assume that $f\in \mathcal{C}^{2}(\bR^d)$ and $\norm{\nabla^2 f(x)}\leq M,\ \forall~x\in\bR^d$, as in Assumption~\ref{Assump:Hess_bdd}. The Jacobian matrix of $\phi(\omega^e)$ at $x$ is
\begin{equation*}
    J(x) = I - \alpha e_i e_i^\top \nabla^2 f(x),
\end{equation*}
which is always invertible by the Sherman-Morison formula if $\alpha < 1/M$. Therefore, according to the inverse function theorem, at any $x\in\bR^d$, $\phi(\omega^e)$ is locally invertible and the inverse is also $\mathcal{C}^1$. Moreover, $\alpha < 1/M$ yields that
\begin{equation*}
    \lim_{\norm{x} \to +\infty} \norm{\phi(\omega^e)(x)} = +\infty,
\end{equation*}
which implies that $\phi(\omega^e)$ is proper, i.e., the preimage of a compact set is still compact. One can thus apply the Hadamard's global inverse function theorem ~\cites{guillemin2010differential, krantz2002implicit} and conclude that $\phi(\omega^e)$ is globally invertible.

\subsection{Validation of Assumption~\ref{asp:SMT}}
\label{sec:validate_asp_SMT}

Suppose that Assumption~\ref{Assump:Hess_bdd} is satisfied and that $\alpha < 1/M$. Then $\varphi(t,\omega^e)$ is a $\mathcal{C}^{1}$ random dynamical system since $\phi(\omega^e)\in \mathcal{C}^1(\bR^d)$. Using the same argument as in ~\cite{chen2023global}, it can be verified that $A^H(\omega^e)$ and $A^H(\omega^e)^{-1}$ are uniformly bounded in $\omega^e\in\Omega^e$ and hence that the linearized system $\Phi^H(t,\omega^e)$ satisfies the conditions of the multiplicative ergodic theorem (Theorem~\ref{thm:MET}). Therefore, the first condition in Assumption~\ref{asp:SMT} is satisfied.

For the second condition in Assumption~\ref{asp:SMT}, the residual of the linearization at $x^*=0$ and $t = 1$ can be computed as
\begin{align*}
    F(1,\omega^e,x) & = \varphi(1,\omega^e, x) - \Phi^H(1,\omega^e) = \phi(\omega^e) - A^H(\omega^e) \\
    & = \left( x - \alpha e_i e_i^\top \nabla f(x)\right) - \left( x - \alpha e_i e_i^\top H x\right) = \alpha e_i e_i^\top (H x - \nabla f(x)),
\end{align*}
where $i = \pi_0(\omega^e)$. Note that $x\mapsto H x - \nabla f(x)$ is a $\mathcal{C}^1$ map since $f\in\mathcal{C}^2(\bR^d)$. As in assumption~\ref{Assump:Hess_bdd}, there exists a neighborhood $N(x^*)$ of $x^*=0$ and constants $B_0, B_1, L$ so that 
\begin{gather*}
    \sup_{x\in N(x^*) } \norm{D_x^k F (1,\omega^e,x)}\leq B_k,\quad\forall~ 0 \leq k \leq 1 ,\ \omega^e \in \Omega^e, \\
    \norm{D_x F(1,\omega^e,x)} \leq \alpha L \norm{x},\quad\forall~ x\in N(x^*),\ \omega^e \in \Omega^e.
\end{gather*} 
This verifies \eqref{eq:asp_SMT} as constants are tempered random variables.

\subsection{Proof of Proposition~\ref{prop:converexpo}} \label{sec:pf_converexpo}

This subsection proves Proposition~\ref{prop:converexpo}, which only uses the original probability space $(\Omega,\mathcal{F},\bP)$ that is the product space of $(\Omega_t,\Sigma_t,\bP_t)$ for all $t\in\bN$, not the extended one $(\Omega^e,\mathcal{F}^e,\bP^e)$. We will use the filtration $\{\mathcal{F}_t\}_{t\in\bN}$ where $\mathcal{F}_t$ is the sigma algebra generated by $\big\{\prod_{s=0}^t B_s \times \prod_{s=t+1}^{+\infty} \Omega_s : B_s\in \Sigma_s,\ s=0,1,\dots,t\big\}$.

Recall that Proposition~\ref{prop:converexpo} states that for any strict saddle point $x^* \in\crits$ with non-degenerate Hessian $\nabla f(x^*)$,
there exists $\widetilde{\Theta}\subset\Theta= \bR^d\times\Omega$ with $\mu(\Theta \setminus\widetilde{\Theta}) = 0$, such that for any $(x_0,\omega)\in \widetilde{\Theta}$, $x_t(\omega) \to x^*$ as $t\to+\infty$ implies that $x_t(\omega) \to x^*$ exponentially. The proof is divided into two parts. Firstly, we construct $\Theta_1\subset \Theta$ with $\mu(\Theta \setminus\Theta_1) = 0$, so that for any $(x_0,\omega)\in \Theta_1$, $x_t(\omega) \to x^*$ implies that $f(x_t(\omega)) \to f(x^*)$ exponentially. Secondly, we construct $\Theta_2\subset \Theta$ with $\mu(\Theta \setminus\Theta_2) = 0$ and define $\widetilde{\Theta} = \Theta_1 \cap \Theta_2$. It can be shown that for any $(x_0,\omega)\in \widetilde{\Theta}$, the exponential convergence of $f(x_t(\omega))$ to $f(x^*)$ implies that $x_t(\omega) \to x^*$ exponentially. These two parts are elaborated in Section~\ref{sec:Theta1} and Section~\ref{sec:Theta2}, respectively. In the rest of this subsection, we always assume that $x^*=0$ and $f(x^*) = 0$ without loss of generality.

 \subsubsection{Exponential convergence of $f(x_t(\omega))$}\label{sec:Theta1}
 For any $(x_0,\omega)\in\Theta$, where $\omega = (i_0,i_1,i_2,\dots)$, and any $t\in\bN$, we define
\begin{equation}\label{eq:It}
    I_t(x_0, \omega)=\begin{cases}
        1,&\text{if } |e_{i_t}^{\top} \nabla f(x_t)| \geq \frac{1}{\sqrt{d}} \norm{\nabla f(x_t)},\\
    0,&\text{otherwise},
    \end{cases}
\end{equation}
where $x_t = x_t(\omega)$ is generated by Algorithm~\ref{alg:RCGD} given $(x_0,\omega)$. We further define
\begin{equation}\label{eq:area1}
    \Theta_1 = \left\{(x_0,\omega)\in \Theta: \liminf_{t \to +\infty} \frac{\sum_{s=0}^{t-1} I_s(x_0, \omega)}{t} \geq \frac{1}{d}\right\}.
\end{equation}

\begin{lemma}\label{lem:area1}
    It holds that $\mu(\Theta \setminus \Theta_1) = 0$.
\end{lemma}

\begin{proof}
    Define 
    \begin{equation*}
        J_t(x_0,\omega) = \begin{cases}
        1,&\text{if } i_t = \min \left\{i: |e_{i}^{\top} \nabla f(x_t)| \geq \frac{1}{\sqrt{d}} \norm{\nabla f(x_t)}\right\} ,\\
    0,&\text{otherwise}.
    \end{cases}
    \end{equation*}
    It is clear that $I_t(x_0,\omega) \geq J_t(x_0,\omega)$ and that
    \begin{equation} \label{argmax}
        \mathbb{P}\left( J_t(x_0,\omega) = 1\;\big\vert\; \mathcal{F}_{t-1}\right)= \frac{1}{d}.
  \end{equation}
    which is because that $\bigl\lvert e_{i_t}^{\top} \nabla f(x_t)\bigr\rvert \geq \frac{1}{\sqrt{d}} \norm{\nabla f(x_t)}$ is true if the $i_t$-th entry of $\nabla f(x_t)$ has the largest absolute value among all entries of $\nabla f(x_t)$, and that $i_t$ is sampled uniformly randomly from $\{1,2,\dots,d\}$.
    Consider
    \begin{equation*}
        \widetilde{J}_t(x_0,\omega) = J_t(x_0,\omega) - \frac{1}{d},
    \end{equation*}
    and we can have that
    \begin{equation*}
        \bE\left(\widetilde{J}_t(x_0,\omega) \;\big\vert\; \mathcal{F}_{t-1}\right) = 0,
    \end{equation*}
    which indicates that $\big\{\sum_{s=0}^t\widetilde{J}_s(x_0,\omega)\big\}_{t\in\bN}$ is a martingale with respect to $\{\mathcal{F}_t\}_{t\in\bN}$. According to the strong law of large numbers for martingales ~\cite{csorgHo1968strong}, for any $x_0\in\bR^d$, we have for $\bP$-almost every $\omega$ that
    \begin{equation*}
        \lim_{t \to +\infty} \frac{\sum_{s=0}^{t-1} \widetilde{J}_s(x_0, \omega)}{t} = 0,
    \end{equation*}
    which implies that
    \begin{equation*}
        \liminf_{t \to +\infty} \frac{\sum_{s=0}^{t-1} I_s(x_0, \omega)}{t}  \geq \lim_{t \to +\infty} \frac{\sum_{s=0}^{t-1} J_s(x_0, \omega)}{t} = \lim_{t \to +\infty} \frac{\sum_{s=0}^{t-1} \widetilde{J}_s(x_0, \omega)}{t} + \frac{1}{d} = \frac{1}{d},
    \end{equation*}
    For any $x_0\in\bR^d$, we set
    \begin{equation*}
        \Omega_1(x_0) = \left\{\omega\in \Omega: \liminf_{t \to +\infty} \frac{\sum_{s=0}^{t-1} I_s(x_0, \omega)}{t} \geq \frac{1}{d}\right\},
    \end{equation*}
    which satisfies
    \begin{equation*}
        \bP(\Omega_1(x_0)) = 1.
    \end{equation*}
    One can thus conclude by applying the Fubini's theorem that  
    \begin{equation*}
        \mu(\Theta \setminus \Theta_1) = \int_{\bR^d} \int_{\Omega} \mathbf{1}_{\Theta \setminus \Theta_1}(x_0,\omega) \mathbb{P}(\mathrm{d}\omega) \textsc{Leb} (\mathrm{d}x_0) = \int_{\bR^d} \mathbb{P}(\Omega\setminus\Omega_1(x_0)) \textsc{Leb} (\mathrm{d}x_0) = 0,
    \end{equation*}
    where $\mathbf{1}_{\Theta \setminus \Theta_1}(x_0,\omega)$ is the indicator function, i.e., $\mathbf{1}_{\Theta \setminus \Theta_1}(x_0,\omega) = 1$ if $(x_0,\omega)\in\Theta\setminus\Theta_1$ and $\mathbf{1}_{\Theta \setminus \Theta_1}(x_0,\omega) = 0$ otherwise. The proof is hence completed.
\end{proof}

\begin{lemma}\label{lem:fconvrate}
    Suppose that Assumption~\ref{Assump:Hess_bdd} holds and that $0<\alpha<1/M$. Let $x^* = 0\in\crits$ be a strict saddle point of $f$ with non-degenerate Hessian $\nabla^2 f(x^*)$ and let $f(x^*) = 0$.
    For $(x_0,\omega) \in \Theta_1$, if $x_t(\omega) \to 0$ as $t\to+\infty$, then $f(x_t(\omega)) \to 0$ exponentially as $t\to+\infty$.
\end{lemma}

\begin{proof}
    Consider any $(x_0,\omega) \in\Theta_1$ with $x_t = x_t(\omega)\to 0$. With Assumption~\ref{Assump:Hess_bdd} and $0<\alpha<1/M$, $f(x_t)$ is monotonically decreasing in $t$, which can be derived by Taylor's expansion at $x_t$,  
    \begin{equation}\label{taylor}
      \begin{split}
            & f(x_{t+1}) = f\left(x_t-\alpha e_{i_{t}}e_{i_{t}}\trans \nabla f(x_t)\right) \\
            & \qquad = f(x_t) - \alpha \left(e_{i_t}\trans \nabla f(x_t)\right)^2  + \frac{1}{2}\alpha^2 \left(e_{i_t}\trans \nabla f(x_t)\right)^2 \cdot e_{i_t}\trans \nabla^2 f\left(x_t-\theta_t \alpha e_{i_{t}}e_{i_{t}}\trans \nabla f(x_t)\right) e_{i_t} \\
            & \qquad \leq f(x_t) - \frac{1}{2} \alpha \left(e_{i_t}\trans \nabla f(x_t)\right)^2 \leq f(x_t),
      \end{split}
    \end{equation}
    where $\theta_t\in(0, 1)$. 
    The monotonicity implies that
    \begin{equation*}
        f(x_t)\geq f(0) = 0,\quad \forall~t\in\bN.
    \end{equation*}
    
    Note that $\nabla^2 f(0)$ is non-degenerate. There exist a neighborhood $U_1$ of $0$ and a constant $\sigma >0$ such that
  \begin{equation}\label{eq:U1}
      \norm{\nabla f(x)}\geq \sigma\norm{x},\quad \forall\ x\in U_1.
  \end{equation}
    Since $x_t \to 0$, we have $x_t \in U_1,\ \forall~t\geq T$ for some $T\in\bN$. Therefore, for any $t\geq T$, if $I_t(x_0, \omega) = 1$, we have from \eqref{eq:It}, \eqref{taylor}, and \eqref{eq:U1} that
    \begin{equation}\label{eq:decay_fxt_It=1}
        f(x_{t+1}) \leq f(x_t) - \frac{\alpha}{2d} \norm{\nabla f(x_t)}^2 \leq f(x_t) - \frac{\alpha\sigma^2}{2d} \norm{x_t}^2 \leq \left(1- \frac{\alpha  \sigma^2}{Md}\right)f(x_{t}),
    \end{equation}
    where the last inequality follows from the Tayler expansion at $0$, namely
    \begin{align*}
        f(x_t) = f(0) + \nabla f(0)\trans x_t +  \frac{1}{2}x_t^{\top} \nabla^2 f(\theta_t' x_t)x_t = \frac{1}{2}x_t^{\top} \nabla^2 f(\theta_t' x_t) x_t \leq \frac{M}{2} \norm{x_t}^2,
    \end{align*}
    with $\theta_t'\in(0,1)$. Applying \eqref{eq:decay_fxt_It=1} and the monotonicity of $f(x_t)$ repeatedly, one obtains that
    \begin{equation*}
        f(x_{t}) \leq \left(1- \frac{\alpha  \sigma^2}{Md}\right)^{\sum_{s=T}^{t-1} I_s(x_0, \omega) } f(x_T), \quad\forall~t\geq T.
    \end{equation*}

    Therefore, one can immediately conclude the exponential convergence rate of $f(x_t)\to 0$ as the construction of $\Theta_1$, say \eqref{eq:area1}, suggests that
    \begin{equation*}
        \liminf_{t \to +\infty} \frac{\sum_{s=T}^{t-1} I_s(x_0, \omega)}{t} = \liminf_{t \to +\infty} \frac{\sum_{s=0}^{t-1} I_s(x_0, \omega)}{t} \geq \frac{1}{d},
    \end{equation*}
    which completes the proof.
\end{proof}

\subsubsection{Exponential convergence of $x_t(\omega)$}\label{sec:Theta2}

We work with the same settings as in Lemma~\ref{lem:fconvrate}, and let $U_1, \sigma$ from \eqref{eq:U1}. Let $\rho  = \rho(\alpha, d, \sigma, M) \in (0, 1)$ be another constant satisfying
\begin{equation}\label{eq:def_r}
    \rho M+\frac{M \rho^2}{2} < \frac{\alpha  \sigma^2}{2d} (1-\rho)^2.
\end{equation}
Define
\begin{equation*}
        U_2 = U_2(\rho) = \bigcup_{x\in f^{-1}(0)} B(x, \rho\norm{x}),
\end{equation*}
where $B(x, r)$ is the open ball in $\bR^d$ centered at $x$ with radius being $r$, and
\begin{equation*}
    S = U_1  \cap U_2 \cap f^{-1}([0,+\infty)).
\end{equation*}
Then we further define that
\begin{equation}\label{eq:area2}
    \Theta_2 = \big\{(x_0,\omega)\in\Theta : x_t(\omega) \in S \text{ finitely often} \big\}.
\end{equation}

\begin{lemma}\label{lem:area2}
    Suppose that Assumption~\ref{Assump:Hess_bdd} holds and that $0<\alpha<1/M$. Let $x^* = 0\in\crits$ be a strict saddle point of $f$ with non-degenerate Hessian $\nabla^2 f(x^*)$ and let $f(x^*) = 0$. It holds that $\mu(\Theta\setminus\Theta_2)=0$.
\end{lemma}

\begin{proof}
    Fix $x_0\in\bR^d$. For any $k\in\bN_+$, define a random variable $\tau_k(x_0)$ via
    \begin{equation*}
        \tau_k(x_0) =\tau_k(x_0,\omega) = \begin{cases}
            t, & \text{if }x_t(\omega) \in S \text{ 
             and }\#\{0\leq t'<t : x_{t'}(\omega)\in S\} = k - 1, \\
             +\infty, &\text{if } \#\{t'\in\bN : x_{t'}(\omega)\in S\} < k,
        \end{cases}
    \end{equation*}
    i.e., $\tau_k$ is the stopping time that $x_t$ visits $S$ for exactly $k$ times. By the definition of $\Theta_2$, i.e., \eqref{eq:area2}, it is clear that $(x_0, \omega) \in \Theta\setminus\Theta_2$ if any only if $\tau_k(x_0,\omega) < +\infty,\ \forall~k\in\bN$.

    Suppose that $x_t\in S$. Since $x_t\in U_2$, there exists $x_t'$ with $f(x_t') = 0$ and $x_t\in B(x_t',\rho\norm{x_t'})$. It is clear that $x_t'\neq 0$ since otherwise $B(x_t',\rho\norm{x_t'}) = \emptyset$, and that $\norm{x_t} \geq (1-\rho)\norm{x_t'}$.
    Using Taylor's expansion at $x_t'$ and Assumption~\ref{Assump:Hess_bdd}, we have
    \begin{align*}
        f(x_t) & \leq \norm{\nabla f(x_t')(x_t-x_t')} + \frac{M}{2}  \norm{x_t-x_t'}^2 \\
    & \leq M \norm{x_t'}\norm{x_t-x_t'}+\frac{M}{2}  \norm{x_t-x_t'}^2 \leq \left(\rho M+\frac{M\rho^2}{2}\right) \norm{x_t'}^2,
    \end{align*}
    which combined with \eqref{eq:decay_fxt_It=1} yields that, with probability at least $\frac{1}{d}$ conditioned on $x_t$, it holds that
    \begin{align*}
        f(x_{t + 1}) \leq f(x_t) - \frac{\alpha \sigma^2}{2d} \norm{x_t}^2 \leq \left(\rho M+\frac{M\rho^2}{2}\right) \norm{x_t'}^2 - \frac{\alpha \sigma^2}{2d} (1-\rho)^2 \norm{x_t'}^2 < 0,
    \end{align*}
    where we used \eqref{eq:def_r}. This proves that as long as $x_t\in S$,
    \begin{equation*}
        \bP\big(f(x_{t+1}) < 0 \ |\ x_t\big) \geq \frac{1}{d},
    \end{equation*}
    which implies that
    \begin{equation*}
        \bP\big(f(x_{t+1}) \geq 0\ |\ \tau_k(x_0) = t\big) \leq 1-\frac{1}{d}.
    \end{equation*}

    Note that $f(x_{t+1}) < 0$ implies that $x_{t'}\notin S$ for any $t'\geq t+1$ due to the monotonically decreasing property \eqref{taylor}.
     For any $k\geq 1$, one can estimate that 
    \begin{align*}
        & \bP\big(\tau_{k+1}(x_0) < +\infty \ |\ \tau_k(x_0)<+\infty\big) \\
        &\qquad =  \sum_{t\in\bN} \bP\big(\tau_k(x_0) = t, \tau_{k+1}(x_0) < +\infty \ |\ \tau_k(x_0) < +\infty\big)\\
        &\qquad \leq \sum_{t\in\bN} \bP\big(\tau_k(x_0) = t, f(x_{t+1}) \geq 0 \ |\ \tau_k(x_0) < +\infty\big) \\
        &\qquad = \sum_{t\in\bN} \bP\big(\tau_k(x_0) = t \ |\ \tau_k(x_0)<+\infty\big) \cdot \bP\big(f(x_{t+1}) \geq 0\ |\ \tau_k(x_0) = t\big) \\
        & \qquad \leq \left(1 - \frac{1}{d}\right) \sum_{t\in\bN} \bP\big(\tau_k(x_0) = t \ |\ \tau_k(x_0) < +\infty\big) \\
        & \qquad = 1 - \frac{1}{d}.
    \end{align*}
    Therefore,
    \begin{align*}
        & \bP\big(x_t\in S,\ \text{i.o.}\ |\ x_0\big) \leq \bP\big(\tau_k(x_0) < +\infty\big) \\
        &\qquad \leq \bP\big(\tau_1(x_0) < +\infty\big)\prod_{k' = 1}^{k-1} \bP\big(\tau_{k'+1}(x_0) < +\infty \ |\ \tau_{k'}(x_0)<+\infty\big) \\
        &\qquad \leq \left(1 - \frac{1}{d}\right)^{k-1},\quad \forall~k\geq 1,
    \end{align*}
    which implies that
    \begin{equation*}
        \bP\big(x_t\in S,\ \text{i.o.}\ |\ x_0\big) = 0,\quad\forall~x_0\in\bR^d.
    \end{equation*}
    One can then conclude $\mu(\Theta\setminus\Theta_2) = 0$ by integrating against $x_0\in\bR^d$.
\end{proof}

We need the following lemma characterizing some local geometric properties near $x^*$.

\begin{lemma}\label{lem:convexesti}
    Suppose that Assumption~\ref{Assump:Hess_bdd} holds and that $0<\alpha<1/M$. Let $x^* = 0\in\crits$ be a strict saddle point of $f$ with non-degenerate Hessian $\nabla^2 f(x^*)$ and let $f(x^*) = 0$. There exists a constant $p>0$ and a neighborhood $U$ of $x^*$, such that
    \begin{equation}\label{eq:fx_x2}
        f(x) \geq p \norm{x}^2,\quad \forall~x\in \left( U\cap f^{-1}([0,+\infty)) \right) \setminus S.
    \end{equation}
\end{lemma}

With Lemma~\ref{lem:convexesti}, one can directly conclude the exponential convergence of $x_t(\omega)$ to $x^*$ from the exponential convergence of $f(x_t(\omega))$ to $f(x^*)$, assuming $x_t(\omega)\to x^*$ and $x_t(\omega)\in S$ finitely often. This leads to the proof of Proposition~\ref{prop:converexpo}.

\begin{proof}[Proof of proposition~\ref{prop:converexpo}]
    Assume $x^*=0$ and $f(x^*)=0$ without loss of generality and define $\widetilde{\Theta} = \Theta_1\cap \Theta_2$. It follows from Lemma~\ref{lem:area1} and Lemma~\ref{lem:area2} that $\mu(\Theta\setminus\widetilde{\Theta}) = 0$. For any $(x_0,\omega) \in \widetilde{\Theta}$ with $x_t(\omega) \to 0$, we have from Lemma~\ref{lem:fconvrate} that $f(x_t(\omega)) \to 0$ exponentially. The definition of $\Theta_2$ in \eqref{eq:area2} and the monotonically decreasing property \eqref{taylor} guarantee that $x_t(\omega)\in ( U\cap f^{-1}([0,+\infty)) ) \setminus S$ for large enough $t$. Then one can conclude the exponential convergence of $x_t(\omega)$ to $0$ by applying~\eqref{eq:fx_x2}.
\end{proof}

We prove Lemma~\ref{lem:convexesti} in the rest of this subsection, where we denote $H = \nabla^2 f(0)\in\bR^{d\times d}$ that is symmetric with all eigenvalues being nonzero, and set
\begin{equation}\label{eq:fH}
    f^H(x) = \frac{1}{2}x^{\top}Hx.
\end{equation}

\begin{lemma}\label{lem:pH}
    Let $H\in \bR^{d\times d}$ be symmetric and consider the quadratic function $f^H$ as in \eqref{eq:fH}. For any $\rho^H > 0$, define
    \begin{equation*}
        U_2^H = U_2^H(\rho^H) = \bigcup_{x\in (f^H)^{-1}(0)} B(x, \rho^H\norm{x}).
    \end{equation*}
    There exist constants $p^H_+>0$ and $p^H_- < 0$ so that the followings hold for any $x\in \bR^d\setminus U_2^H$:
    \begin{itemize}
        \item[(i)] If $f^H(x) \geq 0$, then $f^H(x)\geq p_+^H \norm{x}^2$.
        \item[(ii)] If $f^H(x) \leq 0$, then $f^H(x)\leq p_-^H \norm{x}^2$.
    \end{itemize}
\end{lemma}

\begin{proof}
    We only prove (i) since (ii) is a direct corollary of (i) for $\frac{1}{2}x^\top(-H) x$, and we assume that $(f^H)^{-1}([0,+\infty)) \setminus U_2^H$ is not empty since otherwise the result is trivial. It can be seen that $(f^H)^{-1}(0)$, $\bR^d\setminus U_2^H$, and $(f^H)^{-1}([0,+\infty))$ are all closed under scalar multiplication, which implies for any $c,c'>0$,
    \begin{equation*}
        \frac{c}{c'} \cdot \left(\left( \partial B(0, c')\cap (f^H)^{-1}([0,+\infty)) \right) \setminus U_2^H \right) =  \left( \partial B(0, c)\cap (f^H)^{-1}([0,+\infty)) \right) \setminus U_2^H,
    \end{equation*}
    where $\partial B(0, c')$ and $\partial B(0, c)$ are the boundaries of $B(0,c')$ and $B(0, c)$, respectively. This homogeneity property leads to that
    \begin{equation}\label{eq:homogeneous}
        \inf_{x \in \left( \partial B(0,c')\cap (f^H)^{-1}([0,+\infty)) \right) \setminus U_2^H} \frac{f^{H}(x)}{\norm{x}^2} = \inf_{x \in \left( \partial B(0,c)\cap (f^H)^{-1}([0,+\infty)) \right) \setminus U_2^H} \frac{f^{H}(x)}{\norm{x}^2}.
    \end{equation}
    Notice that $\partial B(0,c)$ and $(f^H)^{-1}([0,+\infty))$ are both closed and that $U_2^H$ is open. Therefore, $\left( \partial B(0,c)\cap (f^H)^{-1}([0,+\infty)) \right) \setminus U_2^H$ is closed, on which $f^H(x) / \norm{x}^2$ is always positive. One can thus conclude that
    \begin{equation}\label{eq:p>0}
        p_+^H:= \inf_{x \in \left( \partial B(0,c)\cap (f^H)^{-1}([0,+\infty)) \right) \setminus U_2^H} \frac{f^H(x)}{\norm{x}^2} = \min_{x \in \left( \partial B(0,c)\cap (f^H)^{-1}([0,+\infty)) \right) \setminus U_2^H} \frac{f^H(x)}{\norm{x}^2} > 0.
    \end{equation}
    Combining \eqref{eq:homogeneous} and \eqref{eq:p>0}, we know that
    \begin{equation*}
        f^H(x) \geq p_+^H \norm{x}^2,\quad \forall~x\in (f^H)^{-1}([0,+\infty)) \setminus U_2^H,
    \end{equation*}
    which proves (i).
\end{proof}

\begin{lemma}\label{lem:U2H}
    Let $x^* = 0\in\crits$ be a strict saddle point of $f$ with
    non-degenerate Hessian $H = \nabla^2 f(x^*)$ and let $f(x^*) = 0$. If
    $H$ has at least one positive eigenvalue and $\rho^H < \rho / 4 <
    1/4$, then there exists a neighborhood $U'$ of $0$, such that
    $U_2^H(\rho^H) \cap U'\subseteq U_2(\rho) \cap U' $.
\end{lemma}

\begin{proof}
    Without loss of generality, we assume that $H = \text{diag}(h_1,\dots,h_{d'}, h_{d'+1},\dots,h_d)$, where $h_1\geq\cdots\geq h_{d'}> 0 > h_{d'+1}\geq \dots \geq h_d$, since otherwise one can change the coordinates via an orthogonal transformation.
    Define
    \begin{equation}\label{eq:constant_U2H}
        c = \frac{1}{2} \min\{ h_{d'},  -h_{d'+1} \} \cdot \left(2\rho^H + (\rho^H)^2\right) > 0,
    \end{equation}
    and let $c'>0$ depend on $c$ and $f$ near $0$ so that 
    \begin{equation}\label{eq:taylor_U2H}
        \left| f(x) - f^H(x) \right| < \frac{c}{4} \norm{x}^2,\quad \forall~x\in B(0, 3c').
    \end{equation}

    Set $U' = B(0, c')$.
    Consider any $x\in U_2^H(\rho^H) \cap U'$, and it suffices to show that $x\in U_2(\rho)$. There exists $x'\in (f^H)^{-1}(0) \setminus \{0\}$ so that $x\in B(x',\rho^H \norm{x'})$ by the definition of $U_2^H(\rho^H)$. One also has $x' \in B(0, 2c')$ since $\rho^H < 1/2$ and $x\in B(0, c')$. Define
    \begin{equation*}
        x_+' = x' + \rho^H P_+ x',\quad x_-' = x' + \rho^H P_- x',
    \end{equation*}
    where $P_+ = \text{diag}(1,\dots, 1, 0, \dots, 0)\in\bR^{d\times d}$ has $d'$ nonzero diagonal entries and $P_- = I - P_+$. Note that 
    \begin{equation*}
        0 = f^H(x') = f^H(P_+ x') + f^H(P_- x'),
    \end{equation*}
    which implies that
    \begin{align*}
        f^H(P_+ x') = - f^H(P_- x') & = \frac{1}{2} f^H(P_+ x') - \frac{1}{2} f^H(P_- x') \geq \frac{h_{d'}}{2} \norm{P_+ x'}^2 - \frac{h_{d'+1}}{2} \norm{P_- x'}^2 \\
        & \geq \frac{1}{2} \min\{ h_{d'},  -h_{d'+1} \} \norm{x'}^2.
    \end{align*}
    One can hence obtain that 
    \begin{equation*}
        f^H(x_+') = (1+\rho^H)^2 f^H(P_+ x') + f^H(P_- x')  = \left(2\rho^H + (\rho^H)^2\right) f^H(P_+ x') \geq c \norm{x'}^2,
    \end{equation*}
    with $c$ being the constant in \eqref{eq:constant_U2H},
    and similarly that
    \begin{equation*}
        f^H(x_-') = f^H(P_+ x') + (1+\rho^H)^2 f^H(P_- x')  = \left(2\rho^H + (\rho^H)^2\right) f^H(P_- x') \leq - c \norm{x'}^2.
    \end{equation*}
    Notice also $x_+', x_-'\in B(x',\rho^H \norm{x'})\subseteq B(0, 3c')$. It holds that
    \begin{equation*}
        f^H(x_+') \geq \frac{c}{4} \norm{x_+'}^2,\quad f^H(x_-') \leq - \frac{c}{4} \norm{x_-'}^2.
    \end{equation*}
    which combined with \eqref{eq:taylor_U2H} yields that
    \begin{equation*}
        f(x_+') > 0 > f(x_-').
    \end{equation*}
    Therefore, there exists $x''\in B(x',\rho^H \norm{x'})$ so that 
    \begin{equation*}
        f(x'') = 0.
    \end{equation*}
    We also have that
    \begin{equation*}
        \norm{x - x''} \leq \norm{x - x'} + \norm{x' - x''} < 2\rho^H \norm{x'} < 4\rho^H \norm{x''} < \rho \norm{x''},
    \end{equation*}
    which leads to that
    \begin{equation*}
        x \in B(x'', \rho\norm{x''})\subseteq U_2(\rho).
    \end{equation*}
    The proof is thus completed.
\end{proof}

\begin{proof}[Proof of Lemma~\ref{lem:convexesti}]
    Let $\rho^H\in(0,\rho / 4)$ and denote $U_2^H = U_2^H(\rho^H)$. By Lemma~\ref{lem:U2H}, there exists a neighborhood $U'$ of $0$, such that $U_2^H\cap U'\subseteq U_2\cap U'$. Then by Lemma~\ref{lem:pH}, there exist constants $p^H_+>0$ and $p^H_- < 0$ so that for $x\in U'\setminus U_2 \subseteq U'\setminus U_2^H$, one has either $f^H(x)\geq p_+^H \norm{x}^2$ or $f^H(x)\leq p_-^H \norm{x}^2$.
    Let $U \subseteq U_1\cap U'$ be a neighborhood of $0$ so that
    \begin{equation*}
        \left|f(x) - f^H(x)\right| < p \norm{x}^2,\quad\forall~x\in U,
    \end{equation*}
    where $p = \frac{1}{2}\min\{p_+^H, -p_-^H\}$. Therefore, we have for any $x\in U\setminus U_2$ that, either $f(x)\geq p\norm{x}^2$ or $f(x)\leq -p\norm{x}^2$.

    Consider any $x\in \left( U\cap f^{-1}([0,+\infty)) \right) \setminus S$. Note that $S = U_1 \cap U_2 \cap f^{-1}([0,+\infty))$ and that $U\subseteq U_1$. We have $x\in U\setminus U_2$ and $f(x)\geq 0$, which excludes $f(x)\leq -p\norm{x}^2$ and leads to $f(x)\geq p\norm{x}^2$.
\end{proof}

\bibliographystyle{amsxport}
\bibliography{references}

\end{document}